\documentclass[12pt,a4paper]{article}

\usepackage[latin1]{inputenc}
\usepackage[english]{babel}
\usepackage{amsmath}
\usepackage{color}
\usepackage{amsfonts}
\usepackage{amssymb,amsthm}
\usepackage{hyperref} 
\usepackage{graphicx,abstract}
\usepackage{cleveref}

\newcommand{\RR}{\mathbb{R}}

\newcommand{\NN}{\mathbb{N}}

\newtheorem{theorem}{Theorem}
\newtheorem*{theorem*}{Theorem}
\newtheorem{lemma}{Lemma}
\newtheorem{proposition}{Proposition}
\newtheorem{corollary}{Corollary}

\newtheorem{claim}{Claim}
\newtheorem{definition}{Definition}

\newtheorem{remark}{Remark}

\renewenvironment{proof}[1][]{\noindent {\bf Proof #1:\;}}{\hfill $\Box$}

\title{On the nature of Bregman functions}
\author{Edouard Pauwels\thanks{Toulouse School of Economics, France. Institut universitaire de France (IUF).}}

\begin{document}
\date{Draft of \today}

\maketitle

\begin{abstract}
				Let $C \subset \RR^n$ be convex, compact, with nonempty interior and $h$ be Legendre with domain $C$, continuous on $C$. We prove that $h$ is Bregman if and only if it is strictly convex on $C$ and $C$ is a polytope. This provides insights on sequential convergence of many Bregman divergence based algorithm: abstract compatibility conditions between Bregman and Euclidean topology may equivalently be replaced by explicit conditions on $h$ and $C$. This also emphasizes that a general convergence theory for these methods (beyond polyhedral domains) would require more refinements than Bregman's conditions.
\end{abstract}

\section{Introduction}

Let $C \subset \RR^n$ be convex compact with nonempty interior, $h \colon C \to \RR$ be convex. We assume throughout the article that $h$ is Legendre, has domain $C$, and is continuous on $C$:
\begin{itemize}
		\item $h$ is continuous on $C$\footnote{The domain of $h$ is $C$ and $h$ is not defined outside $C$. $h$ is assumed to be continuous on $C$ equiped with the subspace topology. Equivalently we could consider $h \colon \RR^n \to \RR \cup \{+\infty\}$, with value $+\infty$ outside $C$, in this case $h_C$, the restriction of $h$ to $C$, is assumed to be continuous. Our arguments are limited to sequences in $C$ so that both points of view are equivalent in this work.}.
		\item $h$ is essentially smooth: continuously differentiable on the interior of its domain, $\mathrm{int} C$, such that for all $x \in \mathrm{bd} C$ and $y \in \mathrm{int} C$, $$\lim_{t \to 0,\ t>0} \left\langle \nabla h(x + t(y-x)), y - x \right\rangle = - \infty$$ 
		\item $h$ is strictly convex on the interior of its domain, $\mathrm{int} C$.
\end{itemize}
Note that this corresponds to the classical definition in \cite[Section 26]{rockafellar1970convex} for the pair $(\mathrm{int} C, h)$, and contrary to \cite{rockafellar1970convex}, $C$ is closed. This convention is used throughout the text, the notation $C$ denotes a closed set which is the focus of our interest.
The Bregman divergence associated to $h$ is then for all $y \in C$ and $x \in\mathrm{int} C$, 
\begin{align*}
				D_h(y,x) = h(y) - h(x) - \left\langle \nabla h(x), y-x \right\rangle.
\end{align*}

\subsection{Fej\'erian sequences:}

Let $f \colon C \to \RR$ be convex proper lower semi-continuous. Consider the problem
\begin{align}
    \min_{x \in C} f(x) 
    \label{eq:mainProblem}
\end{align}
and denote by $S \subset C$, the solution set of \eqref{eq:mainProblem}. The following definition is adapted from the Euclidean setting \cite{combettes2001fejer}, see also the extension to Bregman monotone sequences \cite{bauschke2003bregman}, note that we do not put emphasis on monotonicity, but rather on convergence which is sufficient for our purpose. 
\begin{definition}
				Let $(x_k)_{k \in \NN}$ be a sequence in $\mathrm{int} C$, it is called weak $D$-Fej\'er for problem \eqref{eq:mainProblem} if
				\begin{itemize}
								\item All accumulation points are in $S$.
								\item For all $y \in S$, $D_h(y,x_k)$ has a finite limit as $k \to \infty$.
				\end{itemize}
				\label{def:fejerSequence}
\end{definition}
Actually, \Cref{def:fejerSequence} may be given for a general closed convex set $S \subset C$ (see \textit{e.g.} \cite{combettes2001fejer} which may be directly adapted and \cite{bauschke2003bregman}). We stick to the optimization problem \eqref{eq:mainProblem} for simplicity since it already cover a wide range of algorithms.
\subsection{Algorithmic examples}
\label{sec:examples}
We list below some of the most common algorithms exhibiting Bregman Fej\'erian behaviors in convex optimization.
\paragraph{Mirror descent:} the algorithm is due to Nemirovsky \cite{newmirovsky1983problem}.
Assume that $f$ is Lipschitz. Initialize $x_0 \in \mathrm{int}C$ and, given sequence of positive step sizes $(\alpha_k)_{k \in \NN}$, iterate 
\begin{align}
		x_{k+1} &= \nabla h^* ( \nabla h(x_k) - \alpha_k v_k),\\
		v_k			& \in \partial f(x_k).
    \label{eq:mainMirror}
\end{align}
Assuming $\sum_{k=0}^{+\infty} \alpha_k = + \infty$ and $\sum_{k=0}^{+\infty} \alpha_k^2 < + \infty$ the resulting algorithmic sequence is weak $D$-Fej\'er. In this case equation (4.22) in \cite{beck2003mirror} ensures that accumulation points are in $S$ and equation (4.21) combined with [Lemma 2, Section 2.2]\cite{polyak1987introduction} ensures that $D(y,x_k)$ has a limit for all $y \in S$.
\paragraph{Bregman gradient and NoLips algorithm:} Assume that $f$ is $C^1$ on an open set containing $C$ and $h-\alpha f$ is convex for some $\alpha > 0$. 
Initialize $x_0 \in \mathrm{int}C$ and iterate
\begin{align}
    x_{k+1} = \nabla h^* ( \nabla h(x_k) - \alpha \nabla f(x_k)).
    \label{eq:mainAlgo}
\end{align}
In this case, for all $y \in S$, $D_h(y,x_k)$ is non increasing as $k$ grows \cite[Lemma 5]{bauschke2016descent}. This extends to the NoLips algorithm proposed for composite objectives in \cite{bauschke2016descent}.
Similarly, the work presented in \cite{alvarez2004hessian} describes a continuous time variant of Bregman Fej\'erian properties in the context of Hessian-Riemannian gradient flows.

\paragraph{Proximal minimization with $D$-functions:}
Assume that $f$ is proper lower-semicontinuous. Initialize $x_0 \in \mathrm{int} C$ and, given sequence of positive step sizes $(\alpha_k)_{k \in \NN}$, iterate
\begin{align}
		x_{k+1} = \arg\min_{x \in C} \alpha_k f(x) + D_h(x,x_k).
    \label{eq:mainPMD}
\end{align}
In this case, for all $y \in S$, $D_h(y,x_k)$ is non increasing as $k$ grows \cite[Lemma 3.3]{chen1993convergence} and if $\sum_{k=0}^{+\infty} \alpha_k = + \infty$, then the resulting sequence has all its accumulation point in $S$ \cite[Theorem 3.4]{chen1993convergence} so that the resulting sequence is weak $D$-Fej\'er. 

\paragraph{Alternating projection:}
Beyond optimization problem \eqref{eq:mainProblem}, in his foundational paper \cite{bregman1967relaxation}, Bregman considers an alternating projection algorithm generating a weak $D$-Fej\'er sequence to find an element in the intersection of convex sets. 

\subsection{Sequential convergence analysis} 
The weak $D$-Fej\'er property can be used to prove convergence of the sequence similarly as in the Euclidean case, in relation to Opial's Lemma (see \cite{combettes2001fejer} for an overview). This argument is valid provided that the topology encoded by $D_h$ is equivalent to the Euclidean topology. This is always true in the interior of $C$ by strict convexity, but further assumptions need to be made to ensure that this also holds at the boundary. The first condition is 
\begin{align}
				(x_k)_{k \in \NN} \subset \mathrm{int}C,\, y \in C, \qquad D_h(y,x_k) \underset{k \to \infty}{\to} 0 \quad \Rightarrow \quad x_k \underset{k \to \infty}{\to} y \tag{A}
    \label{eq:convA}
\end{align}
We know from continuity of $h$ on its domain that strict convexity of $h$ on the whole domain $C$ (not only the interior) is sufficient for \eqref{eq:convA} \cite[Lemma 2.16]{kiwiel1997free}.
If we assume the opposite implication
\begin{align}
				(x_k)_{k \in \NN} \subset \mathrm{int}C,\, y \in C,	  \qquad x_k \underset{k \to \infty}{\to} y \quad \Rightarrow \quad D_h(y,x_k) \underset{k \to \infty}{\to} 0  \tag{B}
    \label{eq:convB}
\end{align}
then we have the following F\'ejer argument: assume that \eqref{eq:convA} and \eqref{eq:convB} hold true, then any weak $D$-Fej\'er sequence $(x_k)_{k \in \NN}$ (\Cref{def:fejerSequence}) converges. Indeed, an accumulation point $\bar{x}$ exists by compacity, it must be in $S$ by \Cref{def:fejerSequence}, from property \eqref{eq:convA}, up to a subsequence, $D_h(\bar{x},x_k) \to 0$, but \Cref{def:fejerSequence} ensures that $D_h(\bar{x},x_k)$ converges so its limit must be $0$ and property \eqref{eq:convB} ensures that $x_k \to \bar{x}$.

Note that by continuity of $h$ on its domain, condition \eqref{eq:convB} is equivalent to $\left\langle \nabla h(x_k), y - x_k \right\rangle \to 0$ as $x_k \to y$.
Conditions \eqref{eq:convA} and \eqref{eq:convB} date back to Bregman \cite{bregman1967relaxation} and have been extensively considered in the literature \cite{censor1981iterative,censor1992proximal,eckstein1993nonlinear,chen1993convergence,bauschke1997legendre,
kiwiel1997free,bauschke2003bregman,alvarez2004hessian,auslender2006interior,bauschke2016descent,sorin2023continuous} in the same abstract form or with adaptation to broader settings than considered here. 
Given $h$, a Legendre function, continuous on its compact domain $C$, if conditions \eqref{eq:convA} and \eqref{eq:convB} are satisfied, then $h$ is called a Bregman function. 

\subsection{Main results}
We are interested in the following question: How restrictive are conditions \eqref{eq:convA} and \eqref{eq:convB}?
In other words, how much does it take for a continuous Legendre function to be Bregman? We provide the following answer. 
\begin{theorem}
	Let $C \subset \RR^n$ be convex, compact, with nonempty interior and $h$ be Legendre with domain $C$, continuous on $C$. Then: 
	\begin{itemize}
					\item \eqref{eq:convA} holds if and only if $h$ is strictly convex on $C$.
					\item \eqref{eq:convB} holds if and only if $C$ is a polytope.
	\end{itemize}
	\label{th:mainTheorem}
\end{theorem}
\Cref{th:mainTheorem}, provides an explicit sufficient condition on $h$ ensuring convergence of all algorithms described in \Cref{sec:examples}. Indeed if $C$ is a polytope and $h$ is strictly convex on $C$, then any weak $D$-Fej\'er sequence converges. Furthermore, Theorem \ref{th:mainTheorem} illustrates the fact that existing Fej\'erian arguments for sequential convergence are only valid for polytopic domains, and convergence analysis for more general domains will require different arguments. 

\Cref{th:mainTheorem} can be obtained by combining \Cref{th:domainPolytope}, \Cref{th:polyhedraSufficient}, \Cref{lem:kiwiel} and \Cref{lem:kiwielReverse}. We make crucial use of \cite[Lemma 1]{tseng1991relaxation}.
Let us mention that Theorem \ref{th:mainTheorem} and all our presentation is limited to compact $C$, but most proof arguments relate to possibly unbounded closed $C$ so that the intermediate results convey information of independent interest for this more general case.
The link between strict convexity and \eqref{eq:convA} is essentially known, one implication is due to Kiwiel \cite{kiwiel1997free} and the reverse implication is connected to \textit{total convexity}, see \cite[Proposition 1.2.6]{butnariu2000totally}. We provide a self contained proof based on a result of \cite{tseng1991relaxation} for completeness. The connection between \eqref{eq:convB} and the polytopic nature of $C$ is the most interesting part of our results and we will start with it. We discuss extensions of our main result in \Cref{sec:extensions}.

\section{Condition \eqref{eq:convB}}

Let us first illustrate failure of \eqref{eq:convB} and its relation with curvature with a simple example, which was described independently in \cite[Example 4.2]{azizian2022rate}.

\subsection{Intuition: incompatibility with curvature}

Set $h \colon x \mapsto  - \sqrt{1 - \|x\|^2}$ on $C \subset \RR^2$, the unit Euclidean ball in the plane. We have $\nabla h \colon x \mapsto x \frac{1}{\sqrt{1 - \|x\|^2}}$. Considering polar coordinates in the plane, $x = (r \cos(\theta), r\sin(\theta))$, with $e_1$ the first basis vector $(1,0)$, we have
\begin{align*}
    h(e_1) - h(x) - \left\langle \nabla h(x), e_1 - x\right\rangle &=-\sqrt{1 - r^2} + \frac{r(r - \cos(\theta))}{\sqrt{1 - r^2}}
\end{align*}
Choosing $\theta(r) = \arccos \left( r-  \sqrt{1 - r^2}\right)$ for $r\geq0$, as $r \to 1$, $x(r) = (r \cos(\theta(r)), r\sin(\theta(r)))$ goes to $e_1$ and $D_h(e_1, x) \to 1$.

\begin{figure}[h]
				\centering
				\includegraphics[width=.4\textwidth]{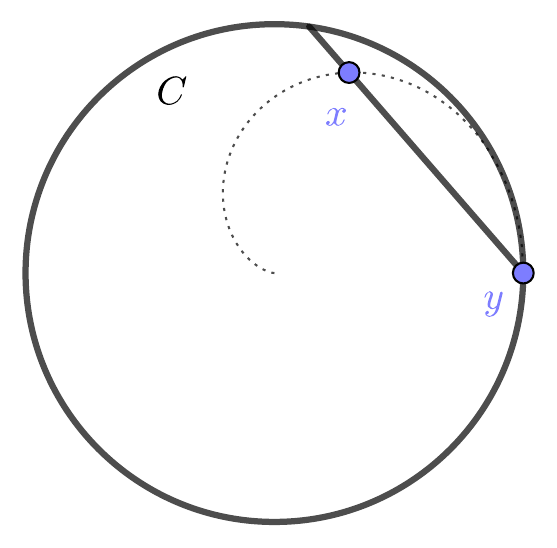}\qquad
				\includegraphics[width=.4\textwidth]{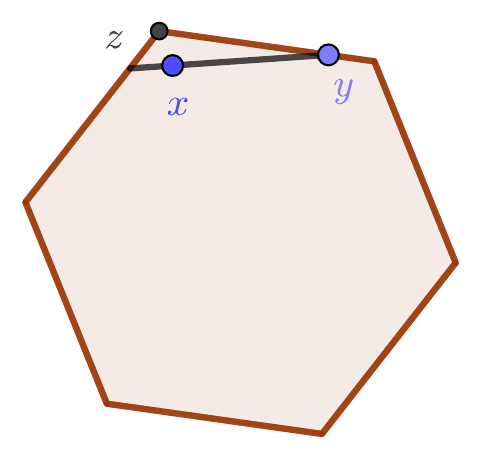}
				\caption{Left: Illustration of the tangential phenomenon, closedness to the boundary and to the opposite extremity of the chord. Right: in general extreme points constitute a region of high curvature and it is possible to find neighboring points in a similar configuration. This cannot happen too much under condition \eqref{eq:convB} and in particular extreme points should not accumulate anywhere.} 
				\label{fig:illustrTangent}
\end{figure}

Actually, the curve $r \to x(r)$ converges to $e_1$ with a vertical tangent as shown in \Cref{fig:illustrTangent}.
This illustrates the main mechanism of failure of condition \ref{eq:convB}. Because of curvature, chords all meet the interior of $C$ so that the directional derivative of $h$ along chords explodes at their endpoints. Here the curve remains on a chord segment in a region of negative directional derivative along the chord.

From this illustration, one intuition is that the boundary of the compact domain $C$ of a Bregman function should not have too much curvature. More precisely it should not have curvature accumulating anywhere. This intuition is actually correct, the key mechanism is that around extreme points, there is a lot of curvature which allows to generate behavior similar to the circle example above, see \Cref{fig:illustrTangent}. This cannot happen too densely otherwise this contradicts \eqref{eq:convB}, and as a result, the extreme points of $C$ have to be isolated, this is expressed in \Cref{th:domainPolytope}.
On the other hand, polyhedra have a very strong structure which will enforce \eqref{eq:convB}  this is \Cref{th:polyhedraSufficient}.

\subsection{First implication}
In this section we prove the following
\begin{proposition}
				Let $C \subset \RR^n$ be convex closed with non-empty interior and $h$ be Legendre, with domain $C$, continuous on $C$ and satisfy \eqref{eq:convB}. Then the extreme points of $C$ are locally finite. In particular if $C$ is bounded it is a polytope.
				\label{th:domainPolytope}
\end{proposition}

We start with two lemmas
\begin{lemma}
				Under the assumptions of \Cref{th:domainPolytope}, let $x\in \mathrm{bd} C$ and $y \in C$ be such that $\frac{x + y}{2} \in \mathrm{int} C$. Then $D_h(y, (1-\lambda)x + \lambda y) \to + \infty$ as $\lambda \to 0$.
				\label{lem:explodingBoundary}
\end{lemma}
\begin{proof}
			Note that in this case the open segment $(x,y)$ is contained in the interior of $C$.
			For any $\lambda \in (0,1)$, we have
			\begin{align*}
							&D_h(y, (1-\lambda)x + \lambda y) - h(y) + h((1-\lambda)x + \lambda y) \\
							=\,& - (1-\lambda)\left\langle \nabla h(x + \lambda (y-x)), y-x \right\rangle\\
							=\, &  - (1-\lambda)\left\langle \nabla h\left(x + 2\lambda \left(\frac{y+x}{2}-x\right)\right), 2 \left(\frac{y+x}{2}-x\right)\right\rangle.
			\end{align*}
			Letting $\lambda \to 0$, the right hand side goes to $+ \infty$ by essential smoothness and the result follows because $ h((1-\lambda)x + \lambda y)$ remains bounded by continuity of $h$ on $C$.
			This concludes the proof.
\end{proof}
\begin{remark}
	\Cref{lem:explodingBoundary} holds true if one relaxes continuity of $h$ by boundedness on $C$. This is discussed in \Cref{sec:extensions}
	\label{rem:hbounded}
\end{remark}

\begin{lemma}
				Under the hypotheses of \Cref{th:domainPolytope}, let $a \in \mathrm{bd} C$ be an extreme points and $y \in C$, different from $a$. Then for any $\epsilon,K > 0$, there exists $c \in \mathrm{int} C$ such that 
				\begin{align*}
								&\|c - a\|\leq \epsilon \\
								&D_h(y, c)  \geq K.
				\end{align*}
				\label{lem:extPtSegment}
\end{lemma}
\begin{proof}
				First if $[a,y] \cap \mathrm{int} C \neq \emptyset$, then we are in the conditions of \Cref{lem:explodingBoundary}, which provides the desired result (with $x=a$). We may therefore assume that $[a,y] \subset \mathrm{bd}C$.  

				\begin{claim}
								There exists $\epsilon_1 > 0$ such that for any $b \in \mathrm{int}C \cap B_{\epsilon_1}(a)$, the line from $y$ to $b$ crosses the boundary of $C$ at $x$ such that $\|x - a\| \leq \epsilon$.
				\end{claim}
				\paragraph{Proof of the claim.} We may assume that $\epsilon_1 < \|y-a\|$, so that, since $a \neq y$, the points $y$ and $b$ define a unique line for any $b \in B_{\epsilon_1}(a)$. We also impose that $\epsilon_1 < \epsilon$. Assume toward a contradiction that for all such $\epsilon_1>0$, there is $b \in \mathrm{int}C \cap B_{\epsilon_1}(a)$, such that the line from $y$ to $b$ either does not cross the boundary of $C$ or it crosses it at a point at distance greater than $\epsilon$. In both cases, since $b \in B_\epsilon (a)$ (as we assumed $\epsilon_1 < \epsilon$), the line from $y$ to $b$ exits $B_\epsilon(a)$ at a point $x \in C$. In this case, using a vanishing sequence of values for $\epsilon_1$, we can produce a sequence $(b_k)_{k \in \NN}$ converging to $a$ such that each segment $[y,b_k]$ can be extended up to $x_k \in C$ where $b_k \in [x_k,y]$ and $\|x_k - a\|= \epsilon$. Passing to the limit, up to subsequences $b_k \to a$ and $x_k \to x \in C$, we have $a \in [x,y]$ where $a \neq y$ and $a \neq x$ which contradicts the fact that $a$ is an extreme point of $C$. This proves the claim.

				We may choose $\epsilon_1 \leq \epsilon$, now consider any point $b \in \mathrm{int}C \cap B_{\epsilon_1}(a)$, $b$ belongs to a segment of the form $[x,y]$ for some $x \in \mathrm{bd} C \cap B_{\epsilon}(a)$. Note that $[b,x] \subset B_{\epsilon}(a)$ so that we can apply \Cref{lem:explodingBoundary} and obtain $c \in [b,x]$ with $D_h(y,c)$ arbitrarily large. This proves the desired result.

\end{proof}

\begin{proof}[of \Cref{th:domainPolytope}]
				Assume that the set of extreme points of $C$ it is not locally finite. This means that we can find a bounded sequence $(z_k)_{k\in\NN}$ of pairwise distinct extreme points. This sequence has a converging subsequence, let $y$ be its limit. Since the sequence has pairwise distinct elements, there is at most one $k \in \NN$ such that $y = z_k$ and we may remove it from the sequence. In other words $z_k \to y \in C$ as $k \to \infty$ and $z_k \neq y$ for all $k \in \NN$.
				We will show that in this case, condition \ref{eq:convB} is violated. 

				Fix $k \in \NN$, using \Cref{lem:extPtSegment} with $a=z_k$, we may find $x_k \in \mathrm{int}C$ such that $\|x_k - z_k\| \leq \|z_k - y\|$ and $D_h(y,x_k) \geq k$. Now as $k \to \infty$, we have $x_k \to y$ and $D_h(y,x_k) \to + \infty$ which contradicts condition \ref{eq:convB}. This concludes the proof. 
\end{proof}

\subsection{Reverse implication}
\begin{proposition}
				Let $C \subset \RR^n$ be convex closed with non-empty interior and locally polyhedral (represented locally by finitely many affine inequalities, for example a polytope), let $h$ be Legendre, with domain $C$, continuous on $C$, then \eqref{eq:convB} holds.
				\label{th:polyhedraSufficient}
\end{proposition}

The result follows from the following (\cite[Lemma 1]{tseng1991relaxation}).
\begin{lemma}[Tseng and Bertsekas]
	Let $h \colon \RR^p \mapsto \RR \cup +\infty$, be lower semicontinuous and continuous on its domain	$\mathrm{dom}\ h$. Then
	\begin{itemize}
		\item For any $y \in \mathrm{dom}\ h$, there exists a nondegenrate closed ball centered at $y$ such that $\mathrm{dom}\ h \cap B$ is closed.
		\item For any $y \in \mathrm{dom}\ h$, and $d$ such that $y + d \in \mathrm{dom}\ h$ and sequences $x_k \to y$ and $d_k \to d$ such that $x_k + d_k \in \mathrm{dom}\ h$ for all $k$, we have
			\begin{align*}
				\lim\sup_{k \to \infty} h'(x_k,d_k) \leq h'(y,d).
			\end{align*}
			where $h'(y,d) = \lim_{t \to 0,\ t>0} \frac{h(y+td) - h(y)}{t}$ for any $y \in \mathrm{dom}\ h$ and $d \in \RR^p$.
	\end{itemize}
	\label{lem:tsengBertsekas1987}
\end{lemma}

\begin{lemma}
				Let $C$ be a polyhedron and $y \in C$. Then there exists $\epsilon > 0$ such that for any $x \in C \cap B_\epsilon(y)$, $2x - y \in C$. 
				\label{lem:legendrePolytopeStrat}
\end{lemma}

\begin{proof}
				Let $a_1,\ldots,a_m \in \RR^n$ and $b_1,\ldots, b_m \in \RR$ such that $C = \{x | \left\langle a_i, x\right\rangle \leq b_i, \forall i=1,\ldots, m\}$. Fix $y \in C$, and $I \subset \{1,\ldots m\}$ the set of active indices, such that $\left\langle a_i , y \right\rangle = b_i$ if and only if $i \in I$. By continuity of linear functions, there exists $\epsilon > 0$ such that for all $d$ with $\|d\| \leq 2\epsilon$, and all $i \not \in I$, $\left\langle a_i , y+d \right\rangle < b_i$. Now for any $x \in C$ such that $\|x-y\| \leq \epsilon$, we have
				\begin{align*}
								\left\langle a_i, 2x - y\right\rangle &= \left\langle a_i, y + 2(x - y)\right\rangle < b_i & \forall i \not \in I,\\
								\left\langle a_i, 2x - y\right\rangle &= 2\left\langle a_i, x \right\rangle - \left\langle a_i, y\right\rangle \leq 2 b_i - \left\langle a_i, y\right\rangle = 2 b_i - b_i = b_i & \forall i \in I,
				\end{align*}
				which shows that $2x - y \in C$ and concludes the proof.
\end{proof}

\begin{proof}[of \Cref{th:polyhedraSufficient}]
				By continuity of $h$ on $C$, it suffices to show that $$\lim_{x \to y, x \in \mathrm{int} C } \left\langle \nabla h(x), y - x\right\rangle = 0$$ for all $y \in C$. 

				Fix $y \in C$ and consider $(x_k)_{k \in \NN}$, converging to $y$. $C$ can be locally represented by a polyhedron and therefore \Cref{lem:legendrePolytopeStrat} can be applied to $y$ and $C$. Let $\epsilon>0$ be given by \Cref{lem:legendrePolytopeStrat} and assume without loss of generality that $\|y -x_k\| \leq \epsilon$ for all $k \in \NN$. We have for all $k \in \NN$, $x_k + (y - x_k) = y \in C$ so that applying \Cref{lem:legendrePolytopeStrat} with $d_k = y - x_k$, which converges to $0$
				\begin{align*}
								{\lim\sup}_{k \to \infty} \left\langle \nabla h(x_k), y - x_k \right\rangle \leq 	h'(y,0) = 0.
				\end{align*}
				Furthermore, by \Cref{lem:legendrePolytopeStrat}, for all $k \in \NN$, we have $2x_k - y = x_k + (x_k - y) \in C$. Therefore, one may apply \Cref{lem:legendrePolytopeStrat} with $d_k = x_k - y$, which also converges to $0$ to obtain
				\begin{align*}
								{\lim\sup}_{k \to \infty} \left\langle \nabla h(x_k), x_k - y \right\rangle \leq 	h'(y,0) = 0.
				\end{align*}
				which is equivalent to
				\begin{align*}
								{\lim\inf}_{k \to \infty} \left\langle \nabla h(x_k), y - x_k\right\rangle \geq 	0.
				\end{align*}
				We have shown that
				\begin{align*}
								0 \leq {\lim\inf}_{k \to \infty} \left\langle \nabla h(x_k), y - x_k\right\rangle \leq {\lim\sup}_{k \to \infty} \left\langle \nabla h(x_k), x_k - y \right\rangle \leq  0,
				\end{align*}
				so that the limit is $0$. This concludes the proof.
\end{proof}
\section{Condition \eqref{eq:convA}}
One implication follows from \cite[Lemma 2.16]{kiwiel1997free}.
\begin{lemma}[Kiwiel]
				Let $C \subset \RR^n$ be convex closed with non-empty interior, let $h$ be Legendre, with domain $C$, continuous on $C$ and strictly convex on $C$, then \eqref{eq:convA} holds.
				\label{lem:kiwiel}
\end{lemma}
The reverse implication follows from the study of \textit{total convexity} in \cite[Proposition 1.2.6]{butnariu2000totally}.
We provide a self contained proof based on \Cref{lem:tsengBertsekas1987} for completeness.  
\begin{lemma}
				Let $C \subset \RR^n$ be convex closed with non-empty interior, let $h$ be Legendre, with domain $C$, continuous on $C$ such that \eqref{eq:convA} holds, then $h$ is strictly convex on $C$.
				\label{lem:kiwielReverse}
\end{lemma}
\begin{proof}
				Toward a contradiction, assume that $h$ is not strictly convex. This means that there exists $x,y \in C$ such that $x \neq y$ and
				\begin{align}
								h\left(\frac{x+y}{2} \right) = \frac{h(x) + h(y)}{2},
								\label{eq:kiwielReverse1}
				\end{align}
				which implies that $h$ is affine along the segment $[x,y]$, that is
				\begin{align*}
					h\left( \frac{x + y}{2} + t(x - y) \right) = h\left( x \left(t + \frac{1}{2}\right) + y \left(\frac{1}{2} - t\right)  \right)  = \frac{h(x) + h(y)}{2} + t(h(x) - h(y))
				\end{align*}
				for all $t \in [-1/2,1/2]$.
				In particular, setting $z = \frac{x+y}{2}$, we have
				\begin{align}
								h'(z, x-z) = -h'(z, y-z) = \frac{h(x) - h(y)}{2}.
								\label{eq:kiwielReverse2}
				\end{align}
				Now consider $z_0 \in \mathrm{int}C$ and the sequence $(z_k)_{k\in \NN}$ in $\mathrm{int} C$, such that for all $k \in \NN$, $k \geq 1$, $z_k = \frac{1}{k} z_0 + \frac{k-1}{k} z \in C$.

				We set for all $k \in \NN$, $k \geq 1$, $y_k = \frac{2}{k+1} z_0 + \frac{k-1}{k+1} y \in C$. We have for all $k \geq 1$, 
				\begin{align*}
								z_k - x &= \frac{1}{k} z_0 + \frac{k-1}{2k} y +\left( \frac{k-1}{2k} - \frac{2k}{2k}\right)x =\frac{2}{2k} z_0 + \frac{k-1}{2k} y - \frac{k+1}{2k}x\\
								y_k - x &= \frac{2}{k+1} z_0 + \frac{k-1}{k+1} y - \frac{k+1}{k+1}x  = (z_k - x) \frac{2k}{k+1}\\
								y_k - z_k &= y_k - x + x - z_k = (z_k - x) \left(\frac{2k}{k+1} - 1\right)= \frac{k-1}{k+1} (z_k - x).\\
				\end{align*}
				Using \Cref{lem:tsengBertsekas1987}, we have
				\begin{align*}
								{\lim\sup}_{k \to \infty} \left\langle \nabla h(z_k), y_k - z_k   \right\rangle & = {\lim\sup}_{k \to \infty} \frac{k-1}{k+1} \left\langle \nabla h(z_k), z_k - x \right\rangle \\
								&= - {\lim\inf}_{k \to \infty} \left\langle \nabla h(z_k),x - z_k \right\rangle \\
								&\leq h'(z, y-z) = -h'(z, x-z), 
				\end{align*}
				so that
				\begin{align*}
								{\lim\sup}_{k \to \infty} \left\langle \nabla h(z_k), x - z_k \right\rangle&\leq h'(z, x-z) \leq {\lim\inf}_{k \to \infty} \left\langle \nabla h(z_k), x - z_k  \right\rangle
				\end{align*}
				and $\left\langle \nabla h(z_k), x - z_k \right\rangle \to h'(z, x-z)$ as $k \to \infty$.
				We deduce by continuity of $h$ using \eqref{eq:kiwielReverse1} and \eqref{eq:kiwielReverse2}
				\begin{align*}
								\lim_{k \to \infty} & h(x) - h(z_k) - \left\langle \nabla h(z_k), x - z_k\right\rangle =  h(x)-  \frac{h(x) + h(y)}{2}  - \frac{h(x) - h(y)}{2} = 0
				\end{align*}
				So we have that $z_k \to z \neq x$ but $D_h(x, z_k) \to 0$ which shows that condition \eqref{eq:convA} does not hold. This proves the result by contraposition.
\end{proof}

\section{Extensions}
\label{sec:extensions}
The proposed analysis is centered on a compact domain $C$ with $h$ continuous on its domain and $\mathrm{dom}\ h = C$. This allows to convey the main message in a simple form. \Cref{th:mainTheorem} has several direct extensions and calls for a broader discussion.

\subsection{Unbounded domain}
We notice that the arguments of \Cref{th:domainPolytope} and \Cref{th:polyhedraSufficient} do not require boundedness of $C$. Let us also point out that \Cref{th:mainTheorem} has the following consequence.
\begin{corollary}
	Let $C \subset \RR^n$ be convex, closed, with nonempty interior and $h$ be Legendre, with domain $C$, continuous on $C$. Then: 
	\begin{itemize}
		\item \eqref{eq:convB} holds if and only if $C$ is locally polyhedral: for any polytope $P$, $P \cap C$ is a polytope.
	\end{itemize}
	\label{cor:unbounded}
\end{corollary}
\begin{proof}
	If $P$ is contained in a strict affine subspace $A$ of $\RR^n$, then the restriction of $h$ to $A$ satisfy our hypothesis on $C \cap A$. Therefore it suffices to consider the full dimensional setting. Any polytope $P$ with non-empty interior admits a Legendre function $h_P$, with domain $P$, continuous on $P$ \footnote{For example using the well known Boltzman-Shannon entropy applied to the polyhedral representation of $P$. Set $\phi(t) = t \log(t)$ for $t>0$ and $\phi(0) = 0$, then if $P = \{x | \left\langle a_i, x\right\rangle \leq b_i, \forall i=1,\ldots, m\}$, then $h_P \colon x \mapsto \sum_{i=1}^m \phi(b_i - \left\langle a_i, x\right\rangle )$ satisfies the the desired property}. Then $h+h_P$ is Legendre with domain $C \cap P$, continuous on $C \cap P$. According to \Cref{th:mainTheorem}, Condition \eqref{eq:convB} holds true for $h+h_P$ if and only if $C \cap P$ is a polytope. Using the fact that $D_{h+h_P} = D_h + D_{h_P}$ we also notice that Condition \eqref{eq:convB} holds true for $h$ and $C$ if and only if Condition \eqref{eq:convB} holds true for $h+h_P$ and $C \cap P$, for all possible polytopes $P$ (with the construction of $h_P$ as above). This concludes the proof.
\end{proof}

\subsection{Bounded lower-semicontinuous $h$}
A carefull inspection of the proof of \Cref{th:domainPolytope} allows to conclude that the same result holds if $h$ is bounded on $C$, not necessarily continuous. Indeed, the conclusion of \Cref{lem:explodingBoundary} holds in this case (see \Cref{rem:hbounded}) and continuity of $h$ is not used further in the proof. This has the following consequence which shows that continuity of $h$ is essentially a requirement for condition \eqref{eq:convB}.
\begin{corollary}
	Let $C \subset \RR^n$ be convex compact with non-empty interior and $h$ be convex, lower-semicontinuous, Legendre, with domain $C$, bounded on $C$ and satisfy \eqref{eq:convB}. Then $h$ is continuous on $C$.
	\label{cor:boundedH}
\end{corollary}
\begin{proof}
	We deduce from \Cref{th:domainPolytope}, \Cref{rem:hbounded} and the preceeding discussion that $C$ is a polytope. Using the main result of \cite{gale1968convex}, $h$ is therefore upper semi-continuous on $C$ and hence continuous since it was also assumed to be lower-semicontinuous.
\end{proof}

The connection between Condition \eqref{eq:convA} and strict convexity of $h$ could also be discussed in light of potential relaxation of the continuity of $h$, we conjecture that continuity of $h$ is not necessary for the equivalence. 

\subsection{Unbounded $h$}
If $\mathrm{dom}\ h \neq C$, then $D_h(y,x)$ is only defined for $x \in \mathrm{int} C$ and $y \in \mathrm{dom}\ h$. Conditions \eqref{eq:convA} and \eqref{eq:convB} have no meaning if $y \not \in \mathrm{dom}\ h$ and it is unclear how to generalize them. It is also difficult to describe the behavior of the function $h$ outside of its domain, for example what would be the correct strict convexity notion on the boundary of $\mathrm{dom}\ h$. These represent important issues since it could be the case that the target solution set $S$ in \eqref{eq:mainProblem} is not contained in $\mathrm{dom}\ h$. Typical results in this setting relate to complexity estimates \cite{bauschke2016descent}, but in general, the sequential convergence of Bregman type algorithms probably represents a hard problem. One would expect positive results under specific structural assumptions such as barier functions constructed based on polyhedral representation (\textit{e.g.} the well known logarithmic barrier), in the spirit of the convergence of the central path for interior point methods.

\section*{Acknowledgements}
The author would like to thank J\'er\^ome Bolte for his continuous support and fruitful interactions and Jalal Fadili for pointing out the connection between sequential consistency and total convexity. The author also thank the anonymous referee for very relevant comments and suggestions on the first version of this work. The author thanks TSE-P.
This work supported by the AI Interdisciplinary Institute ANITI,  ANR-19-PI3A-0004, Air Force Office of Scientific Research, Air Force Material Command, USAF, FA8655-22-1-7012, ANR Regulia, and ANR Chess.

\end{document}